\documentclass[10pt,a4paper]{llncs}
\usepackage{amsmath}
\usepackage{amsfonts}
\usepackage{amssymb}
\usepackage{tikz}
\usetikzlibrary{arrows,backgrounds,decorations,positioning,fit,petri}
\usepackage{ntheorem}
\usepackage{color}
\usepackage{multirow}
\usepackage{abstract}
\usepackage{pdflscape}

\def\H{\mathcal{H}}
\def\NN{\mathbb{N}}

\def\hs{\hbox to 3mm{}}
\def\hhs{\hbox to 5cm{}}
\def\ss{\smallskip}

\def\wmat{\mathrm{WMat}}

\author{G\'erard H.E. Duchamp\inst{1}, Nguyen Hoang-Nghia\inst{1}, Adrian Tanasa\inst{1,2}}

\authorrunning{G.H.E. Duchamp et al.}

\title{A selection-quotient process for packed word Hopf algebra}

\titlerunning{A selection-quotient packed word Hopf algebra}

\institute{${}^{1}$LIPN, UMR 7030 CNRS, Institut Galil\'ee - Universit\'e Paris 13, Sorbonne Paris Cit\'e, 
99 avenue J.-B. Cl\'ement, 93430 Villetaneuse, France, EU \\ 
${}^{2}$
Horia Hulubei National Institute for Physics and Nuclear Engineering,\\
P.O.B. MG-6, 077125 Magurele, Romania, EU\\
\email{ghed@lipn.univ-paris13.fr}\\ 
\email{nguyen.hoang@lipn.univ-paris13.fr} \\ \email{adrian.tanasa@ens-lyon.org}}

\begin{document}

\maketitle

\begin{abstract}
In this paper, we define a Hopf algebra structure on the vector space spanned by packed words 
using a selection-quotient coproduct. 
We show that this algebra is free on its irreducible packed words. 
Finally, we give some brief explanations on the Maple codes we have used.

\end{abstract}


\section{Introduction}\label{sec:intro}

In computer science, one is led to 
the 
study of  algebraic structures based on trees, graphs, tableaux, matroids, words and other discrete structures. 
Hopf algebras are also shown to play an important role in 
quantum field theory \cite{CK00}, non-commutative QFT \cite{TK}, \cite{TVT08}, 
(see also the review articles
\cite{Tan10a} \cite{Tan12}), 
or in quantum gravity spin-foam models \cite{Mar03}, \cite{Tan10b}. 
These algebras use a selection-quotient rule for the coproduct.
\begin{equation}\label{eq:type1} \Delta(S) = \sum_{\substack{A \subseteq S\\ +\, Conditions}} S[A] \otimes S/_A ,
\end{equation} 
where $S$ is some (general) combinatorial object (tree, graph, matroid, {\it etc.}), 
$S[A]$ is a substructure of $S$ and $S/_A$ is a quotient.


The present article introduces a new Hopf algebraic structure, 
which we call $\mathrm{WMat}$, on the set of packed words.  
The product is given by the shifted concatenation and 
the coproduct is given by such a selection-quotient principle. 




\section{Algebra structure}
\subsection{Definitions}

Let $X$ be an infinite totally ordered alphabet $\{x_i\}_{i \geq 0}$ and $X^*$ be the set of words with letters 
in the alphabet $X$. 

A word $w$ of length $n=|w|$ is a mapping $i\mapsto w[i]$ from $[1..|w|]$ to $X$. For a letter $x_i\in X$, the partial degree $|w|_{x_i}$ is the number of times the letter $x_i$ occurs in the word $w$. One has:
\begin{equation}
	|w|_{x_i}\ =\sum_{j=1}^{|w|} \delta_{w[j],x_i}.
\end{equation}
For a word $w\in X^*$, one defines the alphabet $Alph(w)$ as the set of its letters, while $IAlph(w)$ is the set of indices in $Alph(w)$.
\begin{equation}
Alph(w)=\{x_i|\ |w|_{x_i}\not=0\}\ ;\ IAlph(w)=\{i\in \NN|\ |w|_{x_i}\not=0\}.
\end{equation}
The upper bound $sup(w)$ is the supremum of $IAlph(w)$, i. e. 
\begin{equation}
sup(w) = sup_{\,\mathbb{N}} (IAlph(w)).
\end{equation}
Note that $sup(1_{X^*}) = 0$. 

\ss
Let us define the substitution operators. Let $w=x_{i_1} \dots x_{i_m}$ and 

$\phi : IAlph(w) \longrightarrow \mathbb{N}$, with $\phi(0)=0$. One then defines:
\begin{equation}\label{eq:subs}
S_\phi(x_{i_1} \dots x_{i_m}) = x_{\phi(i_1)} \dots x_{\phi(i_m)}.
\end{equation} 

Let us define the pack operator of a word $w$. 
Let $\{j_1 , \dots ,j_k \} = IAlph(w)\setminus \{0\}$ with $j_1 < j_2 < \dots < j_k$ and define
$\phi_w$ as 
\begin{equation}\label{eq:packedword}
\phi_w(i) = \begin{cases} m \mbox{ if } i = j_m \\ 0 \mbox{ if } i = 0 \end{cases}.
\end{equation}                                                 

The corresponding packed word, denoted by $pack(w)$, is $S_{\phi_w}(w)$. 
This means that if the word $w$ has (one or several) ``gap(s)'' 
between the indices of its letters, then in the word  $pack(w)$ these
gaps have vanished (the indices of the respective letters being modified accordingly).

\begin{example}
Let $w = x_1 x_1 x_5 x_0 x_4$ . One then has $pack(w) = x_1 x_1 x_3 x_0 x_2$.
\end{example}

A word $w \in X^*$  is said to be \textit{packed} if $w = pack(w)$.

\begin{example}
The packed words of weight 2 are $x_0^{k_1} x_1 x_0^{k_2} x_1 x_0^{k_3}$, with $k_1, k_2, k_3 \geq 0$.
\end{example}

The operator $pack: X^* \longrightarrow X^*$ is idempotent ($pack \circ pack = pack$). It defines, by linear extension, a projector. The image, $pack(X^*)$, is the set of packed words.

\ss
Let $u,v$ be two words; one defines the shifted concatenation $*$ by \begin{equation}u*v = uT_{sup(u)}(v),\end{equation} where, for $t \in \NN$, $T_t(w)$ denotes the image of $w$ by $S_\phi$ for $\phi(n) = n+t$ if $n>0$ and $\phi(0) = 0$ (in general, all letters can be reindexed except $x_0$). It is straightforward to check that, in the case the words are packed, the result of a shifted concatenation is a packed word.

\begin{definition}
Let $k$ be a field. One defines a vector space $\mathcal{H} = span_{k}(pack(X^*))$. One can endow this space with a product (on the words) given by 
 \begin{align*}\mu : &\mathcal{H} \otimes \mathcal{H} \longrightarrow \mathcal{H}, \\ & u \otimes v \longmapsto u*v.\end{align*}
\end{definition}

\begin{remark}
The product above is similar to the shifted concatenation for permutations. 
Moreover, if $u,v$ are two words in $X^*$, then $sup(u*v) = sup(u) + sup(v)$. 
\end{remark}

\begin{proposition}
$(\mathcal{H},\mu,1_{X^*})$ is an associative algebra with unit (AAU).
\end{proposition}
\begin{proof}

Let $u,v,w$ be three words in $\mathcal{H}$. One then has:
\begin{align}
(u*v)*w  = u(T_{sup(u)}(v)(T_{sup(u) +sup(v)})(w)) = u* (v*w). 
\end{align}

Thus, the algebra $(\mathcal{H},\mu)$ is associative. On the other hand, for all $u \in pack(X^*)$, one can easily check that
\begin{equation*}u*1_{X^*}  = u = 1_{X^*}*u.\end{equation*}

Now remark that $pack(1_{X^*})=1_{X^*}$. This is clear from the fact that $1_{X^*} = 1_\mathcal{H}$.

One concludes that $(\mathcal{H},\mu,1_{X^*})$ is an AAU.
\qed \end{proof}

As already announced in the introduction, we call this algebra $\wmat$.

\begin{remark}
The product is non-commutative, for example: $x_1*x_1x_1 \neq x_1x_1*x_1$.
\end{remark}

Let $w = x_{k_1} \dots x_{k_n}$ be a word and $I \subseteq [1 \dots n]$. A sub-word $w[I]$ is defined as $x_{k_{i_1}} \dots x_{k_{i_l}}$, where
$i_j \in I$.

\begin{lemma}\label{lm:subword}Let $u,v$ be two words. Let $I \subset [1\dots |u|]$ and $J\subset [|u|+1\dots |u|+|v|]$. One then has \begin{equation}\label{eq:mor1}pack(u*v[I+J]) = pack(u[I])*pack(v[J']),\end{equation} where $J'$ is the set $\{i-|u|\}_{i\in J}$.
\end{lemma}

\begin{proof}
By direct computation, one has:
\begin{eqnarray}\label{eq:prod}
&  pack(u*v[I+J]) = pack(uT_{sup(u)}(v)[I+J]) = pack(u[I]T_{sup(u)}(v)[J])\notag\\ 
& =pack(u[I]T_{sup(u[I])}(v[J'])) = pack(u[I])*pack(v[J']). 
\end{eqnarray}
\qed \end{proof}

\ss
\begin{theorem}
Let $k<X>$ be equipped with the shifted concatenation. The mapping pack \begin{equation}
k<X> \xrightarrow{pack} \mathcal{H}
\end{equation} is then a morphism AAU.
\end{theorem}

\begin{proof}
By using the Lemma \ref{lm:subword} and taking $I=[1 \dots |u|]$ and $J = [|u|+1 \dots |u|+|v|]$, one gets the conclusion.
\qed \end{proof}

\subsection{$\wmat$ is a free algebra}\label{sec:freealg}

$\wmat$ is, by construction, the algebra of the monoid $pack(X^*)$, therefore to check that $\wmat$ is a free algebra, it is sufficient to show that $pack(X^*)$ is a free monoid on its atoms. 

Here we will use an ``internal" characterization of free monoids in terms of irreducible elements.

\begin{definition}\label{def:irrword}
A packed word $w$ in $pack(X^*)$ is called an irreducible word 
if and only if it can not be written under the form $w = u*v$, 
 where $u$ and $v$ are two non trivial packed words.
\end{definition}

\begin{example}
The word $x_1x_1x_1$ is an irreducible word. The word $x_1x_1x_2$ is a reducible word because 
it can be written as $x_1x_1x_2 = x_1x_1 *x_1$.
\end{example}

\begin{proposition}
If $w$ is a packed word, then $w$ can be written uniquely as 
$w = v_1*v_2*\dots *v_n$, where $v_i$, $1\leq i \leq n$, are non-trivial irreducible words.
\end{proposition}

\begin{proof}
The $i^{th}$ position of word $w$ is called an admissible cut if $sup(w[1\dots i]) = inf(w[i+1 \dots |w|]) - 1$ or $sup(w[i+1 \dots |w|]) = 0$, where $inf(w)$ is infimum of $IAlph(w)$. 

Because the length of word is finite, one can get $w = v_1* v_2* \dots * v_n$, 
with $n$ maximal and $v_i$ non trivial, $\forall 1\leq i \leq n$.

One assumes that one word can be written in two ways 
\begin{equation}\label{eq:irre1} w = v_1* v_2* \cdots * v_n\end{equation} 
and 
\begin{equation}\label{eq:irre2} w = v'_1* v'_2* \cdots * v'_m.\end{equation}

Denoting by $k$ the first number such that $v_k \neq v'_k$, without loss of generality, one can suppose that $|v_k|<|v'_k|$. From equation \eqref{eq:irre1}, the $k^{th}$ position is an admissible cut of $w$. From equation \eqref{eq:irre2}, the $k^{th}$ position is not an admissible cut of $w$. One thus has a contradiction. One has $n=m$ and $v_i=v'_i$ for all $1\leq i \leq n$.
\qed \end{proof}

\medskip
One can thus conclude that $pack(X^*)$ is free as monoid with the packed words as a basis.



\section{Bialgebra structure}
  
Let us give the definition of the coproduct and prove that the coassociativity property holds.

\begin{definition}
Let $A \subset X$, one defines $w/A = S_{\phi_A}(w)$ with 

${\phi_A}(i) = \begin{cases} i \mbox{ if } x_i \not\in A,\\ 0 \mbox{ if } x_i \in A.\end{cases}$

Let u be a word. One defines $^w/_u = ^w/_{Alph(u)}$.
\end{definition}

\begin{definition}\label{def:coprod}
The coproduct of $\H$ is given by \begin{equation}\label{eq:coprod}\Delta(w) = \sum_{I+J=[1\dots |w|]}pack(w[I]) \otimes pack(w[J]/_{w[I]}), \forall w \in \H, \end{equation} where this sum runs over all partitions of $[1 \dots |w|]$ divided into two blocks, $I \cup J = [1 \dots |w|]$ and $I \cap J =\emptyset$.
\end{definition}

\begin{example}
One has:
\begin{eqnarray*}
\Delta(x_1x_2x_1) = & x_1x_2x_1 \otimes 1_{X^\ast} + x_1 \otimes x_1x_0 + x_1 \otimes x_1^2 + x_1 \otimes x_0x_1 + x_1x_2 \otimes x_0 \\ & + x_1^2 \otimes x_1 + x_2x_1 \otimes x_0 + 1_{X^\ast}\otimes x_1x_2x_1.
\end{eqnarray*}
\end{example}

\ss
Let us now prove the coassociativity. 

Let $I=[i_1,\dots,i_n]$, and $\alpha$ be a mapping: \begin{align}
\alpha:  I & \longrightarrow [1\dots n],\notag\\  i_s & \longmapsto s.
\end{align}

\begin{lemma}\label{lm:coprod1}
Let $w \in X^*$ be a word, $I$ be a subset of $[1\dots |w|]$ and $I_1 \subset [1 \dots |I|]$. 
One then has
 \begin{equation}\label{eq:coprod1} pack(w[I])[I_1] = S_{\phi_{w[I]}}(w[I'_1]),\end{equation} where $I'_1$ is $\alpha^{-1}(I_1)$ and $\phi_{w[I]}$ is the packing map of $w[I]$ that is given in \eqref{eq:packedword} .
\end{lemma}

\begin{proof}
Using the definition of packing map $\phi_{w[I]}$, one can directly check that equation \ref{eq:coprod1} holds.
\qed \end{proof}

\ss
\begin{lemma}Let $w \in X^*$ be a word and $\phi$ be a strictly increasing map from $IAlph(w)$ to $\mathbb{N}$. One then has:
\begin{itemize}
\item[1)]  \begin{equation}pack(S_{\phi}(w)) = pack(w).\end{equation}
\item[2)] \begin{equation}\label{eq:quotionpack}S_{\phi}(^{w_1}/_{w_2})= ^{S_{\phi}(w_1)}/_{S_{\phi}(w_2)}.\end{equation}
\end{itemize}\label{lm:coprod2}
\end{lemma}

\begin{proof}

1) Using the definition \eqref{eq:packedword} of the packing map, one has the following identity:
\begin{equation*}
pack(S_{\phi}(w))  = S_{\phi_{S_{\phi}(w)}}(S_{\phi}(w))= S_{\phi_{S_{\phi}(w)}\circ\phi}(w) = S_{\phi_w}(w) = pack(w).
\end{equation*}
 
2) Let $I_2 = Alph(w_2)$ and $I'_2 = Alph(S_{\phi}(w_2))$. One can directly check that $\phi\circ\phi_{I_2}(i) = \phi_{I'_2}\circ\phi(i)$, for all $i \in IAlph(w_1)$. One thus has 
\begin{eqnarray}
& S_{\phi}(^{w_1}/_{w_2})= S_{\phi}(S_{\phi_{I_2}}(w_1)) = S_{\phi\circ\phi_{I_2}}(w_1) = \notag\\ &  S_{\phi_{I'_2}\circ\phi}(w_1) = S_{\phi_{I'_2}}(S_{\phi}(w_1)) = ^{S_{\phi}(w_1)}/_{S_{\phi}(w_2)}.
\end{eqnarray}
\qed \end{proof}

\ss
\begin{lemma} Let $w$ be a word in $\mathcal{H}$, and $I,J,K$ be three disjoint subsets of $\{1 \dots |w|\}$. One then has: \begin{equation}
\frac{^{w[K]}/_{w[I]}}{^{w[J]}/_{w[I]}} = ^{w[K]}/_{w[I+J]}.
\end{equation}\label{lm:coprod3}
\end{lemma}

\begin{proof}

Using Lemma \ref{lm:coprod2}, one has: \begin{align}
&\frac{^{w[K]}/_{w[I]}}{^{w[J]}/_{w[I]}} = ^{S_{\phi_I}(w[K])}/_{S_{\phi_I}(w[J])} = S_{\phi_I}(^{w[K]}/_{w[J]}) = S_{\phi_I}(S_{\phi_J}(w[K])) \notag\\& = S_{\phi_I \circ \phi_J}(w[K]) =  ^{w[K]}/_{w[I+J]}.
\end{align}
\qed \end{proof}

\ss
\begin{proposition}
The vector space $\mathcal{H}$ endowed with the coproduct \eqref{eq:coprod} is a coassociative coalgebra with co-unit (c-AAU). The co-unit is given by: 
$$\epsilon(w) = \begin{cases} 1 \mbox{ if } w = 1_\H,\\0 \mbox{ otherwise}.  \end{cases}$$ 
\end{proposition}

\begin{proof}
Using the lemmas \ref{lm:coprod1}, \ref{lm:coprod2} and \ref{lm:coprod3}, one has the following identity:
\begin{eqnarray}\label{eq:LHSassc}
&(\Delta \otimes Id)\circ \Delta(w) = \sum_{I+J+K = [1\dots |w|]} pack(w[I]) \otimes pack(^{w[J]}/_{w[I]}) \notag\\&  \otimes pack(^{w[K]}/_{w[I+J]}) =(Id \otimes \Delta)\circ \Delta(w).
\end{eqnarray}

Thus, one can conclude that the coproduct \eqref{eq:coprod} is coassociative.

One can easily check that \begin{equation}\label{eq:counit}(\epsilon \otimes Id)\circ \Delta(w) = (Id \otimes \epsilon)\circ\Delta(w), \end{equation} for all word $w \in \mathcal{H}$.

One thus concludes that $(\mathcal{H},\Delta,\epsilon)$ is a c-AAU.
\qed \end{proof}

\ss
\begin{remark}
This coalgebra is not cocommutative, for example: 
\begin{align*}
T_{12} \circ \Delta (x_1^2) & = T_{12} (x_1^2 \otimes 1_\H + 2 x_1 \otimes x_0 + 1_\H \otimes x_1^2) \\ & = x_1^2 \otimes 1_\H + 2 x_0 \otimes x_1 + 1_\H \otimes x_1^2 \neq \Delta(x_1^2),
\end{align*}
where the operator $T_{12}$ is given by $T_{12}(u\otimes v) = v \otimes u$.
\end{remark}

\begin{lemma}\label{lm:prod_2}
Let $u,v$ be two words. Let $I_1 +J_1 = [1\dots |u|]$ and $I_2 + J_2 = [|u|+1\dots |u|+|v|]$. One then has
\begin{equation}\label{eq:mor2}
pack(^{u*v[J_1+J_2]}/_{u*v[I_1+I_2]}) = pack(^{u[J_1]}/_{u[I_1]}) * pack(^{v[J'_2]}/_{v[I'_2]}),
\end{equation} where $I'_2$ is the set $\{k-|u|,k\in I_2\}$ and $J'_2$ is the set $\{k-|u|,k\in J_2\}$.
\end{lemma}

\begin{proof}
One has:
\begin{align}
& pack(^{u*v[J_1+J_2]}/_{u*v[I_1+I_2]}) = pack(S_{\phi_{I_1+I_2}}(u*v[J_1+J_2])) \notag\\
& = pack(S_{\phi_{I_1}+\phi_{I_2}}(u[J_1]T_{sup(u)}(v)[J_2]))  = pack(S_{\phi_{I_1}}(u[J_1])S_{\phi_{I_2}}(T_{sup(u)}(v[J'_2]))) \notag\\ 
& = pack(^{u[J_1]}/_{u[I_1]})T_{sup(^{u[J_1]}/_{u[I_1]})}pack(S_{\phi_{I'_2}}(v[J'_2]))  \notag\\
& = pack((^{u[J_1]}/_{u[I_1]})*pack(^{u[J'_2]}/_{u[I'_2]}).
\end{align}
\qed \end{proof}

\ss
\begin{proposition}
Let $u,v$ be two words in $\mathcal{H}$. One has: \begin{equation} \Delta(u*v) = \Delta(u)*^{\otimes 2} \Delta(v).\end{equation}
\end{proposition}

\begin{proof}

Using the Lemma \ref{lm:prod_2}, the proof can be done by a direct check.

\qed \end{proof}

\medskip
Since $\mathcal{H}$ is graded by the word's length, one has the following theorem:

\begin{theorem}
$(\mathcal{H},*,1_\H,\Delta,\epsilon)$ is a Hopf algebra.
\end{theorem}

\begin{proof}
The proof follows from the above results.
\qed \end{proof}

\medskip
For $w \neq 1_\H$, the antipode is given by the recursion: \begin{equation}
S(w) = -w - \sum_{I+J =[1\dots |w|], I,J\neq \emptyset} S(pack(w[I]))*pack(^{w[J]}/_{w[I]}).
\end{equation} 



\section{Hilbert series of the Hopf algebra $\wmat$}

In this section, we compute the number of packed words with length $n$ and supremum $k$. 
Using the formula of Stirling numbers of the second kind, 
one can get the explicit formula for the number of packed words with length $n$, number which we denote by $d_n$.  

\begin{definition} The Stirling numbers of the second kind count the number of set unordered partitions of an $n$-element set into precisely $k$ non-void parts (or blocks). The Stirling numbers, denoted by $S(n,k)$ are given by the recursive definition:

\begin{enumerate}
\item $S(n,n) =1 (n \geq 0)$,
\item $S(n,0) =0 (n > 0)$,
\item $S(n+1,k) = S(n,k-1) + kS(n,k)$, for $0<k\leq n$.
\end{enumerate}
\end{definition}

One can define a word without $x_0$ by its positions, 
this means that if a word $w=x_{i_1}x_{i_2}\dots x_{i_n}$ has length $n$ and 
alphabet $IAlph(w) = \{1,2,\dots,k\}$, then this word can be determine from the list 
$[S_1,S_2,\dots,S_k]$, where $S_i$ is the set of positions of $x_i$ in the word $w$, with $1 \leq i \leq k$.
It is straightforward to check that $(S_i)_{0 \leq i \leq k}$ is a partition of $[1 \dots n]$.

One can divide the set of packed words with length $n$ and supremum $k$ in two parts: \textit{``pure"} packed words (which have no $x_0$ in their alphabet), denote $pack_{n,k}^+(X)$ and packed words which have $x_0$ in their alphabet, denote $pack_{n,k}^0(X)$. It is clear that: \begin{equation}
d(n,k) = \# pack_{n,k}^+(X) + \# pack_{n,k}^0(X).
\end{equation}

Let us now compute the cardinal of these two sets $pack_{n,k}^+(X)$ and $pack_{n,k}^0(X)$.

Consider a word $w \in pack_{n,k}^+(X)$, then $IAlph(w) = \{1,2,\dots,k\}$. This word is determined by $[S_1,S_2,\dots,S_k]$, in which $S_i$ is a set of positions of $x_i$, for $1\leq i \leq k$. One can see that:\begin{enumerate}
\item $S_i \neq \emptyset$, $\forall i \in [1,k]$;
\item $\sqcup_{1 \leq i \leq k} S_i = \{1,2,\dots,n\}$.
\end{enumerate}

Note that 1-2 hold even with $w=1_\H$.

Thus, one has the cardinal of packed words with length $n$ and supremum $k$:
\begin{equation}\label{eq:d+}
d^+(n,k)=\# pack_{n,k}^+(X) = S(n,k)k!.
\end{equation}

Similarly, a 
word $w \in \# pack_{n,k}^0(X)$ can be determined by $[S_0,S_1,S_2,\dots,S_k]$ where $S_i$ is the set of positions of $x_i$, for all $0\leq i \leq k$. 
One then has: 
\begin{equation}\label{eq:d0}
d^0(n,k)=\# pack_{n,k}^0(X) = S(n,k+1)(k+1)!.
\end{equation}

From the two equations above, one can get the number of packed word with length $n$, supremum $k$:
\begin{equation}\label{eq:dnk1}
d(n,k) = d^+(n,k) + d^0(n,k) = S(n,k)k! + S(n,k+1)(k+1)! = S(n+1,k+1)k!.
\end{equation}

From this formula, using Maple, one can get some values of $d(n,k)$. We give in the Table \ref{tab:dnk} the first values.
\begin{table}[!ht]
\caption{Values of $d(n,k)$ given by the explicit formula \eqref{eq:dnk1} and computed with Maple.\label{tab:dnk}}
\begin{center}
\begin{tabular}{ccrrrrrrrrr}\hline
 &&\multicolumn{9}{c}{k}\\
&&0&1&2&3&4&5&6&7&8\\\cline{1-11}
\multicolumn{1}{l}{\multirow{9}{*}{n}}& \multicolumn{1}{l}{0}&1&0&0&0&0&0&0&0&0\\
\multicolumn{1}{l}{}&\multicolumn{1}{l}{1}&1&1&0&0&0&0&0&0&0\\
\multicolumn{1}{l}{}&\multicolumn{1}{l}{2}&1&3&2&0&0&0&0&0&0\\
\multicolumn{1}{l}{}&\multicolumn{1}{l}{3}&1&7&12&6&0&0&0&0&0\\
\multicolumn{1}{l}{}&\multicolumn{1}{l}{4}&1&15&50&60&24&0&0&0&0\\
\multicolumn{1}{l}{}&\multicolumn{1}{l}{5}&1&31&180&390&360&120&0&0&0\\
\multicolumn{1}{l}{}&\multicolumn{1}{l}{6}&1&63&602&2100&3360&2520&720&0&0\\
\multicolumn{1}{l}{}&\multicolumn{1}{l}{7}&1&127&1932&10206&25200&31920&20160&5040&0\\
\multicolumn{1}{l}{}&\multicolumn{1}{l}{8}&1&255&6050&46620&166824&317520&332640&181440&40320\\\cline{1-11}
\end{tabular}
\end{center}
\end{table}

Note that the values of Table \ref{tab:dnk} correspond to those of the triangular array $A028246$ of Sloane \cite{Slo}.

\begin{remark}
 Formulas \eqref{eq:d+} and \eqref{eq:d0} imply that the packed words 
of length $n$ and supremum $k$ without, and respectively with, $x_0$ 
 are in bijection with the circularly ordered partitions of $[n]$ in $k$ parts and respectively in $k+1$ parts. 
 Therefore \eqref{eq:dnk1} implies that the set of packed words of length $n$ with supremum $k$ is in bijection with the circularly ordered partitions of $n+1$ elements in $k+1$ parts.
 \end{remark}

The formula for the number of packed words of length $n$, $d_n$ ($n\geq 1$), is then given by \begin{align}\label{eq:dn}
&d_n =\sum_{k=0}^{n} d(n,k)= \sum_{k=0}^n S(n+1,k+1)k!.
\end{align}

Using again Maple, one can get the values listed in Table \ref{tab:dn}.
\begin{table}[!ht]
\caption{Some values of $d_n$ obtained from formula \eqref{eq:dn}.\label{tab:dn}}
\begin{center}
\begin{tabular}{crrrrrrrrrrrr|}\hline
n&0&1&2&3&4&5&6&7&8&9&10\\\hline
$d_n$&1&2&6&26&150&1082&9366&94586&1091670&14174522&204495126 \\\hline
\end{tabular}
\end{center}
\end{table}

The number of packed words is the sequence 
$A000629$ of Sloane
\cite{Slo}, where it is also mentioned that this sequence corresponds to 
the ordered Bell numbers sequence times two (except for the $0$th order term).

The ordinary and exponential generating function of our sequence are also given in \cite{Slo}. 
The ordinary one is given by 
the formula: $\sum_{n\ge 0} \frac{2^n n! x^n}{\prod_{k=0}^n (1+k x)}$. 
The exponential one 
is given by: $\frac{e^x}{2-e^x}$. Let us give the proof of this.

\ss
Firstly, recall that 
the exponential generating function of the ordered Bell numbers 
(see, for example, page $109$ of Philippe Flajolet's book \cite{FS09}) is:

\begin{equation}\label{eq:pt2}
\frac{1}{2-e^x} = \sum_{n\geq 0 } \sum_{k = 0}^n S(n,k)k! \frac{x^n}{n!}.
\end{equation}

By deriving both side of equation \eqref{eq:pt2} with respect to $x$ , one obtains: 
\begin{equation}\label{eq:pt3}
\frac{e^x}{2-e^x} = \sum_{n\geq 1 } \sum_{k = 1}^n S(n,k)k! \frac{x^{n-1}}{(n-1)!}.
\end{equation}

From equations \eqref{eq:dn} and \eqref{eq:pt3}, one gets the exponential generating function of our sequence:

\begin{equation}\label{eq:egfdn}
\frac{e^x}{2-e^x} = \sum_{n\geq 0 } \sum_{k = 0}^n S(n+1,k+1)k! \frac{x^{n}}{n!} = \sum_{n\geq 0 } d_n \frac{x^{n}}{n!}.
\end{equation}

\bigskip
Let us now investigate the combinatorics of irreducible packed words (see Definition \ref{def:irrword}).
Firstly, we notice that one still has an infinity of irreducible packed words of weight $m$, which are again
obtained by adding multiple copies of the letter $x_0$.

\begin{example}
The word $x_1x_0^kx_1x_0^kx_1$ (with $k$ an arbitrary integer) 
is an irreducible packed word of weight $3$.
\end{example}

Let us denote by $i_n$ the number of irreducible packed words of length $n$. Then one has:
\begin{equation}
i_n =  \sum_{\substack{j_1+ \dots + j_k =n\\j_l \neq 0}}(-1)^{k+1} d_{j_1} \dots d_{j_k}.
\end{equation}

Using Maple, one can get the values of $i_n$, which we give  in Table \ref{tab:irr} below.

\begin{table}[!ht]
\caption{Ten first values of the number of irreducible packed words.\label{tab:irr}}
\begin{center}
\begin{tabular}{crrrrrrrrrrrr}\hline
n&0&1&2&3&4&5&6&7&8&9&10\\\hline
$i_n$ &1 &2 &2 &10 &66 &538 &5170 &59906 &704226 &9671930 &145992338 \\\hline
\end{tabular}
\end{center}
\end{table}

Note that this sequence does not appear in Sloane's On-Line Encyclopedia of Integer Sequences \cite{Slo}.

\section{Primitive elements of $\wmat$}

Let us emphasize that this Hopf algebra, although graded, is not cocommutative and thus the primitive
elements do not generate the whole algebra but only a sub Hopf algebra on which $\Delta$ is cocommutative (the biggest on which CQMM theorem holds).

We denote by $Prim(\wmat)$ the algebra generated by the primitive elements of $\H$.

\ss
Let us recall the following result:
\begin{lemma}\label{lm:graded}
Let $V^{(1)}$ and $V^{(2)}$ be two graded vector space.
\begin{equation}
V^{(i)} = \oplus_{n\geq 0} V_n^{(i)} ,\ i = 1, 2.
\end{equation}                                           
Let $\phi \in  Hom^{gr} (V^{(1)} , V^{(2)} )$, that means $(\forall n \geq 0)(\phi (V^{(1)}_n ) \subseteq V_n^{(2)})$.
Then, $Ker(\phi)$ is graded.
\end{lemma}

\medskip
One then has:
\begin{proposition}
$Prim(\wmat)$ is a Lie subalgebra of $\wmat$, graded by the word's length.
\end{proposition}

\begin{proof}
Let us define the mapping \begin{eqnarray}
\Delta^+:  & \wmat  \longrightarrow \wmat \otimes \wmat \notag\\
& \begin{cases} 1_\H & \longmapsto 0 \\ h & \longmapsto \Delta(h) - 1_\H \otimes h -h \otimes 1_\H \end{cases}.
\end{eqnarray}

This mapping is graded. Using Lemma \ref{lm:graded}, one has $Prim(\wmat) = Ker(\Delta^+)$. Thus, the subalgebra $Prim(\wmat)$ is graded.
\qed \end{proof}

\medskip
Let $\wmat_n$ be the subalgebra generated by the packed words of length $n$, $n \geq 0$. One can see an element $P$ of this subalgebra as a polynomial of packed words of length $n$.

\begin{equation}
P = \sum_{w \in \wmat_n} \left\langle P|w\right\rangle w.
\end{equation} 

Let us now compute the dimensions of the first few spaces $Prim(\wmat)_n$.
For $n=1$, 
one has a basis formed by the primitive elements $x_0$ and $x_1$. 
Then one can check that the primitive elements of length $1$ have the form $ax_0 + bx_1$, with $a$ 
and $b$ scalars. 
%
For $n=2$, one has a basis formed by the primitive elements: $x_0 x_1 - x_1 x_0$ and $x_1 x_2 - x_2 x_1$. 
Then one can check that all the primitive elements of the length $2$ 
have the form $a(x_0 x_1 - x_1 x_0 ) + b(x_1 x_2 - x_2 x_1)$, with $a$ and $b$ scalars.
This comes from explicitly solving a system of $4$ equations with $d_2=6$ variables.

Nevertheless, the explicit calculations quickly become lengthy. Thus, for $n=3$, 
one has to solve a system of $22$ equations with $26$ variables. 



\section{Maple coding}


To test our results with Maple, 
we implement a random word in the following way.
To each word we associate a certain monomial which encodes, using a given alphabet 
the position of any letter and its value. For example,
to the word $x_2 x_3$ we associate the monomial
 $a_1^2a_2^3$ where the powers ($2$ and respectively $3$) correspond
to the values of the letters ($x_2$ and respectively $x_3$)
and the indices ($1$ and $2$) correspond to the positions of the respective letters.

One has to keep in mind that the letter $x_0$ can also be present in the words.
which is encoded with a supplementary word length variable.


Using this idea we can then implement in Maple packed words 
(obtained with a Maple function taking as an argument a general word). 

We have also implemented the LHS (left hand side) and the RHS (right hand side)
of the coproduct formula. For this purpose, one needs to refine the above function
by considering two distinct alphabets to ''build up'' the words, such that one 
can easily separate - as function of the different alphabets - the LHS from the RHS.

Finally, using all of the above, we have checked the coassociativity condition 
for random words up to length $7$, with maximal power $7$.



\section*{Acknowledgements}

We acknowledge Jean-Yves Thibon for various discussions and suggestions.
The authors also acknowledge a Univ. Paris 13, Sorbonne Paris Cit\'e BQR grant.
A. Tanasa further acknowledges
the grants PN 09 37 01 02 and CNCSIS Tinere Echipe 77/04.08.2010.
G. H. E. Duchamp acknowledges the grants ANR BLAN08-2\_332204 (Physique Combinatoire) and PAN-CNRS 177494 (Combinatorial Structures and Probability Amplitudes).



\addcontentsline{toc}{section}{References}
\bibliographystyle{plain}

\end{document}